\documentclass[reqno,11pt]{amsart}
\usepackage{a4wide,color,eucal,enumerate,mathrsfs}
\usepackage[normalem]{ulem}
\usepackage{amsmath,amssymb,epsfig,amsthm} 
\usepackage[latin1]{inputenc}
\usepackage{psfrag}
\usepackage{pdfpages}

\newtheorem{theorem}{Theorem}[section]
\newtheorem{lemma}[theorem]{Lemma}
\newtheorem{prop}[theorem]{Proposition}
\newtheorem{corollary}[theorem]{Corollary}
\newtheorem{definition}[theorem]{Definition}
\numberwithin{equation}{section}

\theoremstyle{remark}

\def\bint{{\ifinner\rlap{\bf\kern.35em--}
\int\else\rlap{\bf\kern.45em--}\int\fi}\ignorespaces}

\title[Sharp stability for critical points of the Sobolev inequality in the absence of bubbling]{Sharp stability for critical points of the Sobolev inequality in the absence of bubbling} 

\author{Gemei Liu and Yi Ru-Ya Zhang}
\date{\today}

\address{ETH Z\"urich, Department of Mathematics, R\"amistrasse 101, 8092, Z\"urich, Switzerland}
\email{gemei.liu@math.ethz.ch}  

\address{State Key Laboratory of Mathematical Sciences, Academy of Mathematics and Systems Science, Chinese Academy of Sciences, Beijing 100190, China}
\address{
Academy of Mathematics and Systems Science, the Chinese Academy of Sciences, Beijing 100190, China}
\email{yzhang@amss.ac.cn}

 \thanks{The first author is grateful to her PhD advisor, Prof. Alessio Figalli, for his guidance and support during the preparation of this manuscript. The second author is funded by National Key R\&D Program of China (Grant No. 2021YFA1003100), the Chinese Academy of Science,  and  NSFC grant No. 12288201. Both authors sincerely appreciate Prof. Giulio Ciraolo and Michele Gatti for identifying a small gap in the original version of the paper.}


\begin{document}

\begin{abstract}
   When $u$ is close to a single Talenti bubble $v$ of the $p$-Sobolev inequality, we show that
   \begin{equation*}
   \|Du-Dv\|_{L^p(\mathbb{R}^n)}^{\max\{1,p-1\}}\le C  \|-{\rm div}(|Du|^{p-2}Du)-|u|^{p^*-2}u\|_{W^{-1,q}(\mathbb{R}^n)},
\end{equation*}
where $C=C(n,p)>0$. This estimate provides a sharp stability estimate for the Struwe-type decomposition in the single bubble case, generalizing the result of Ciraolo, Figalli, and Maggi \cite{CFM2018} (focusing on the case $p=2$) to the arbitrary $p$. Also, in the Sobolev setting, this answers an open problem raised by Zhou and Zou in \cite[Remark 1.17]{ZZ2023}.
\end{abstract}

\maketitle

\section{Introduction}
\subsection{Background}
Given $n\ge 2$ and $1<p<n,$ we denote by $\dot{W}^{1,p}(\mathbb{R}^n)$ the completion of $C_c^\infty(\mathbb{R}^n)$ under a specific seminorm defined as
$$
\|u\|_{\dot{W}^{1, p}(\mathbb{R}^n)}=\left(\int_{\mathbb{R}^n}|Du|^pdx\right)^{\frac{1}{p}}.
$$
Additionally, denote by $W^{-1,q}(\mathbb{R}^n)$  the dual space of $\dot{W}^{1,p}(\mathbb{R}^n)$ with $q=\frac{p}{p-1}.$

Given $p^*=\frac{np}{n-p},$ the $p$-Sobolev inequality states the existence of a largest positive constant $S_{n,\,p}=S(n,p)>0$ such that 
\begin{equation}\label{p-Sob-inequ}
\|Du\|_{L^p(\mathbb{R}^n)}\ge S_{n,\,p}\|u\|_{L^{p^*}(\mathbb{R}^n)}.
\end{equation}
We refer to this constant $S_{n,\,p}$ as the \emph{optimal Sobolev constant}. Furthermore, according to the works of Aubin \cite{A1976} and Talenti \cite{T1976}, the set of all extremal functions for which the inequality in \eqref{p-Sob-inequ} becomes an equality forms a specific $(n+2)$-dimensional manifold given by: 
$$\mathscr M=\left\{u\in {\dot W}^{1,\,p}(\mathbb R^n)\colon u(x)=\frac{a}{\left(1+b|x-x_0|^{\frac{p}{p-1}}\right)^{\frac{n-p}{p}}}, a\in \mathbb{R}\setminus \{0\},b>0,x_0\in \mathbb{R}^n\right\}.$$

On the other hand, up to a scaling factor, the corresponding Euler--Lagrange equation of \eqref{p-Sob-inequ} is 
\begin{equation}\label{EL-equ}
    -\Delta_p v:=-{\rm div}(|Dv|^{p-2}Dv)=|v|^{p^*-2}v \quad  \text{ in }\ \mathbb{R}^n.
\end{equation}
 Let $\mathcal{M}$ be the $(n+1)$-dimensional manifold as follows
\begin{equation*}
    \mathcal{M}=\left\{aU[z,\lambda]\colon U[z,\lambda]=\frac{\lambda^{\frac{n-p}{p}}}{\left(1+\lambda^{\frac{p}{p-1}}|x-z|^{\frac{p}{p-1}}\right)^{\frac{n-p}{p}}},\ \lambda>0,\ z\in \mathbb{R}^n\right\},
\end{equation*}
where $a\in (0,+\infty)$  is chosen such that every $v\in \mathcal M$ satisfies \begin{equation}\label{norm-value}
    \|Dv\|^p_{L^p(\mathbb{R}^n)}=S_{n,\,p}^n\quad \text{and}\quad \|v\|^{p^*}_{L^{p^*}(\mathbb{R}^n)}=S_{n,\,p}^n.
\end{equation}
According to \cite{A1976,CNV2004,T1976}, $\mathcal{M}$ is the space of all positive solutions of \eqref{EL-equ}. Functions in $\mathcal{M}$ are commonly referred to as \emph{Talenti bubbles}. 

Given that the minimizers of \eqref{p-Sob-inequ} and all positive solutions of \eqref{EL-equ} are known, it is indeed natural to investigate their corresponding stability.

\subsection{Stability of the Sobolev energy}
It follows from the concentration-compactness theorem that, when the Sobolev energy 
$$S_{n,\,p}(u):=\frac{\|Du\|_{L^p(\mathbb R^n)}}{\|u\|_{L^{p^*}(\mathbb R^n)}}$$
approaches  $S_{n,\,p}$, then $u$ is close to $\mathscr M$, qualitatively. 

Brezis and Lieb \cite{BL1985} were the first to address the stability problem, aiming to quantify this closeness in the case where $p=2$. Specifically, for a function  $u\in \dot{W}^{1,2}(\mathbb{R}^n)$, the goal is to establish the following quantitative inequality:  
\begin{equation}\label{Sobolev stability energy}
    S_{n,\,p}(u)-S_{n,\,p}\ge c(n,\,p) \omega(d(u,\,\mathscr M)),
\end{equation}
where  $\omega$ is a sharp modulus of continuity, $d(u,\,\mathscr M)$ is an optimal distance from $u$ to $\mathscr M.$
Typically,  $\omega$  is expected to be of the polynomial form   $t^{\alpha}$ with $ \alpha>0$. The challenge usually lies in finding a sharp exponent $\alpha$.

This problem was completely settled by Bianchi and Egnell \cite{BE1991} a few years later. They utilized the Hilbert structure of $\dot W^{1,\,2}(\mathbb R^n)$ to prove the existence of a constant $C_{BE}=C_{BE}(n)>0$ such that, for any $u\in \dot W^{1,\,2}(\mathbb R^n)$,
$$
   S_{n,\,2}(u)-S_{n,\,2}\ge C_{BE} {\inf_{v\in \mathscr M}}\|D (u-v)\|_{L^2(\mathbb{R}^n)}^2.
$$
It is worth noting that the exponent  $2$ on the right-hand side of this inequality is sharp, and the choice of the distance is optimal. Moreover, obtaining an explicit bound of $C_{BE}$ remained an open problem for a long time, primarily due to the limitations imposed by Lion's concentration and compactness principle. This problem was recently settled by  Dolbeault, Esteban, Figalli, Frank, Loss \cite{DEFFL2022}, who established a lower bound for 
$C_{BE}$ using competing symmetries, continuous Steiner symmetrization, and a refined expansion near the manifold of minimizers.

 For general exponents $1<p<n,$ Cianchi, Fusco, Maggi and Pratelli \cite{CFMP2009}  first proved a stability inequality of the form \eqref{Sobolev stability energy}. Their inequality involved an explicit moduli of continuity 
 $\omega=t^{\alpha}$ with $ \alpha=\alpha(n,\,p)>0$ and a distance given by the  $L^{p^*}$-norm of functions  (rather than their gradients).   
Subsequently, Figalli, Maggi, and Pratelli \cite{FMP2013} proved a version of  stability inequality \eqref{Sobolev stability energy} for the special case $p=1.$ 
 
Building upon the results of \cite{CFMP2009}, Figalli and Neumayer \cite{FN2019}, as well as Neumayer \cite{N2020}, established versions of the stability inequality for $p\ge 2$ and $1<p<2$, respectively, with the optimal distance function associated with the $L^p$-norm of the gradient. Unfortunately, the exponents in their inequalities were not optimal. To address this issue, Figalli and Zhang \cite{FZ2022} proved stability with the optimal gradient distance and the sharp exponent  $\max\{p,\,2\}$ for all $1<p<n$ . They achieved this by establishing new vectorial inequalities and proving some corresponding spectral gap estimates.

\subsection{Stability of the Euler Lagrange equation to the Sobolev inequality}

Alternatively, one can perturb the Euler--Lagrange equation instead of the Sobolev energy. Struwe \cite{S1984} presented a fundamental theorem that studied the decomposition of nonnegative solutions that "almost" solve the Euler--Lagrange equation. Later, Mercuri and Willem \cite{MW2010} extended Struwe's result to the general case $1<p<n$. They provided a more comprehensive understanding of the behavior of solutions to the Euler--Lagrange equation. 
We summarize both of their results as follows. 
\begin{theorem}\label{Struwe-decom}
    Let $n\ge 2,$ $1<p<n$, $q$ the H\"older dual of $p$, and $\nu\ge 1$ be positive integers. Assume that $\{u_k\}_{k\in\mathbb{N}}\subset \dot{W}^{1,p}(\mathbb{R}^n)$ is a sequence of nonnegative functions such that $$\left(\nu-\frac{1}{2}\right)S^n\le \int_{\mathbb{R}^n}|Du_k|^p dx \le \left(\nu+\frac{1}{2}\right)S^n$$ and  $$\|\Delta_p u_k+|u_k|^{p^*-2}u_k\|_{W^{-1, q}(\mathbb{R}^n)}\rightarrow 0\quad \text{as}\ k\rightarrow 0.$$
    Then, there exists a sequence $(z_1^{(k)},\cdots,z_\nu^{(k)})_{k\in \mathbb{N}}$ of $\nu$-tuples of points in $\mathbb{R}^n$ and sequences $(\lambda_1^{(k)},\cdots,\lambda_\nu^{(k)})_{k\in \mathbb{N}}, (a_1^{(k)},\cdots,a_\nu^{(k)})_{k\in \mathbb{N}}$  of $\nu$-tuples of positive real numbers such that 
    $$\left\|D\left(u_k-\sum_{i=1}^\nu a_i^{(k)}U[z_i^{(k)},\lambda_i^{(k)}]\right)\right\|_{L^p(\mathbb{R}^n)}\rightarrow 0,\quad \text{as}\ k\rightarrow \infty.$$
\end{theorem}

Similar to the stability of the Sobolev energy, there is considerable interest in quantifying Theorem~\ref{Struwe-decom} for its applications in geometry, particularly in the context of Yamabe flows. When $p=2$, the pioneering work by Ciraolo, Figalli, Maggi \cite{CFM2018} provided a sharp quantitative estimate of Struwe's decomposition in the case where  $\nu=1$. More precisely, they proved  that 
 \begin{equation*}
     \|Du-a_1DU[z,\lambda]\|_{L^2(\mathbb{R}^n)}\le C \|\Delta u+u|u|^{\frac{4}{n-2}}\|_{W^{-1,2}(\mathbb{R}^n)}.
 \end{equation*}
Furthermore, using a localization argument and by computing the interactions between bubbling phenomena, Figalli and Glaudo \cite{FG2020} demonstrated that the linear control obtained in the single bubble case also extends to the multiple bubbles case when $3\le n\le 5$, i.e. 
 \begin{equation*}
\|D u-\sum_{i=1}^{\nu}a_iDU[z_i,\lambda_i]\|_{L^2(\mathbb{R}^n)}\leq C\|\Delta u+u|u|^{\frac{4}{n-2}}\|_{W^{-1,\,2}(\mathbb{R}^n)}.
\end{equation*}
However, they pointed out that, when $n\ge 6$, one cannot expect such a linear control anymore. Despite this challenge, the problem was eventually resolved by the remarkable work of Deng, Sun, and Wei \cite{DSW2021}. They employed a finite-dimensional reduction method to overcome the difficulties associated with higher dimensions and provided a comprehensive solution to the problem in this range: 
\begin{equation*}
\|D u-\sum_{i=1}^{\nu}a_iD U[z_i,\lambda_i]\|_{L^2(\mathbb{R}^n)}\leq C(n,\,p)
\begin{cases}
  \Gamma|\log \Gamma|^{\frac{1}{2}}, & n=6, \\
  \Gamma^{\frac{n+2}{2(n-2)}}, & n\geq 7,
\end{cases}
\end{equation*}
where $\Gamma=\|\Delta u+u|u|^{\frac{4}{n-2}}\|_{W^{-1,\,2}(\mathbb{R}^n)}.$

However, for the case $1<p<n$, a quantified version of Theorem~\ref{Struwe-decom} remained elusive. Indeed, compared to the case $p=2,$ three main challenges arise:\\
\textbf{Lack of Hilbert Space Structure:} When $p=2,$ the Sobolev space $\dot{W}^{1,2}(\mathbb{R}^n)$ is a Hilbert space, allowing one to find $v$ such that $u-v$ is orthogonal to the tangent space of $\mathcal{M}$ at $v$ via projection. However, for general $p,$ the Sobolev space $\dot{W}^{1,p}(\mathbb{R}^n)$ is merely a Banach space. To construct a Hilbert space structure, it is crucial to give a suitable notion of orthogonality. This issue has been addressed in \cite[Section 2]{FN2019}. Additionally, even if $\|Du-Dv\|_{L^p(\mathbb{R}^n)}\le \hat{\epsilon}$ for any $\hat{\epsilon}>0,$ it does not necessarily follow that $v$ is orthogonal to the tangent space of $\mathcal{M}$ at $v.$ This difficulty was overcome in  \cite[Lemma 4.1]{FZ2022}, which ensures the existence of the desired $v\in \mathcal{M}$ by introducing a new minimization principle.\\
\textbf{Nonlinearity of the $p$-Laplacian:}       When $p=2,$ the $p$-Laplacian reduces to the classical Laplacian, enabling a complete decomposition of the main term and the error term due to its linearity. However, for $1<p<n,$ the $p$-Laplacian is nonlinear, which poses challenges in expanding the term $|D u|^{p-2}D u\cdot D (u-v).$ When $p\ge 2,$ as noted in \cite[Remark 1.17]{ZZ2023}, a standard Taylor expansion yields the sharp stability estimates only under certain restrictions on $|Du-Dv|.$ For $p<2,$ as explained in \cite[Section 1.3]{FZ2022}, the $\dot{W}^{1,p}$ norm does not control any weighted $\dot{W}^{1,2}$ norm, preventing a second order expansion of $|D u|^{p-2}D u\cdot D (u-v).$ Moreover, for $p\le \frac{2n}{n+2},$ the $L^{p^*}$ norm fails to control any weighted $L^2$ norm, making the expansion of $|u|^{p^*-2}u(u-v)$ more challenging. These difficulties compel us to find new approaches for expanding $|D u|^{p-2}D u\cdot D (u-v)$ in regimes $p\in (1,2)$ and $p\in [2,n),$ as well as for expanding $|u|^{p^*-2}u(u-v)$ in $p\in (1,\frac{2n}{n+2}]$ and $\ p\in (\frac{2n}{n+2},n).$\\
\textbf{Compactness and Spectral Analysis:} For $p=2,$ compactness results such as the Rellich--Kondrachov Theorem and the spectral analysis of the classical Laplacian are readily available. However, for the case $1<p<n$, extra techniques are required. Inspired by \cite[Lemma 2.1, Lemma 2.4]{FZ2022}, we establish new vectorial inequalities by replacing original weight $|x|^{p-2}$ with suitable $\omega_j\ (j=1,\cdots, 4)$ as in Lemma \ref{key-ineq-p-lem} and Lemma \ref{key-ineq-p*-lem}. Then, to get the sharp stability results, it is crucial to refine the corresponding spectral gap inequalities under this small perturbation. One of the main difficulties in this process is to establish new compact embedding theorems. Fortunately, in our case, the compactness results from \cite[Proposition 3.1, Lemma 3.4, Corollary 3.5]{FZ2022} provide a foundation for deriving the desired spectral gap-type estimates, see Proposition \ref{sepctrum}.

In this manuscript, we focus on studying the stability of the Euler--Lagrange equation associated with the  $p$-Sobolev inequality when $\nu=1$.


\begin{theorem}\label{main-result}
     There exists  $\delta=\delta(n,p)>0$ small enough, such that the following statement holds: If $u\in \dot{W}^{1,p}(\mathbb{R}^n)$ satisfies $\|D u-D\tilde{U}\|_{L^p(\mathbb{R}^n)}\leq \delta,$ where $\tilde{U}\in \mathcal{M},$ then there exist $v\in \mathcal{M}$ and $C=C(n,p)>0,$ such that $u-v$ is orthogonal to $T_v\mathcal{M}$ and 
    \begin{equation*}
   \|Du-Dv\|_{L^p(\mathbb{R}^n)}^{\max\{1,p-1\}}\le C \|P(u)\|_{W^{-1,q}(\mathbb{R}^n)},
\end{equation*}
where $P(u)=-{\rm div}(|Du|^{p-2}Du)-|u|^{p^*-2}u.$
\end{theorem}

Our theorem is also partially motivated by a problem raised in the pre-printed online version of \cite[Remark 1.17]{ZZ2023}. In this intriguing manuscript, the authors studied the stability of the Euler--Lagrange equation associated with the Caffarelli-Kohn-Nirenberg inequality, which can be viewed as a weighted version of the Sobolev inequality. We believe that with suitable modifications, our method has the potential to address their problem as well.  Moreover,  the sharpness of the exponent $1$ when $p<2$ in Theorem~\ref{main-result} follows from a smooth perturbation of $v$, and that of the exponent $p-1$ when $p\ge 2$  follows similarly as the one in \cite[Remark 1.18]{ZZ2023}.

Combining Theorem \ref{main-result} with Theorem \ref{Struwe-decom}, similar to \cite[Corollary 3.4]{FG2020}, we can directly derive the following corollary.
\begin{corollary}
    Let $n\ge 2.$ For any nonnegative function $u\in \dot{W}^{1,p}(\mathbb{R}^n)$ such that  $$\frac{1}{2}S^n\le \int_{\mathbb{R}^n}|Du|^p dx \le \frac{3}{2}S^n,$$
    there exists $C=C(n,p)>0,$ such that
    \begin{equation*}
   \| Du-Dv\|_{L^p(\mathbb{R}^n)}^{\max\{1,p-1\}}\le C\|P(u)\|_{W^{-1,\,q}(\mathbb{R}^n)}.
\end{equation*}
\end{corollary}

The proof of Theorem~\ref{main-result} shares similarities with the one employed in \cite{FZ2022}, but it includes non-trivial modifications tailored to our specific problem. Specifically, instead of relying on \cite[Lemma 2.1]{FZ2022}, we establish a series of new vectorial inequalities, denoted as Lemma \ref{key-ineq-p-lem}, which are compatible with our problem. These inequalities provide an improved version of the vectorial inequalities presented in \cite[Lemma 4.2]{ZZ2023}. Additionally, we prove a result concerning the new spectral gap associated with these vectorial inequalities, as stated in Proposition \ref{sepctrum}. These modifications and additions are crucial for deriving our main theorem and addressing the stability of the Euler--Lagrange equation in the context of the $p$-Sobolev inequality.

While completing this manuscript, we became aware of the recent preprint by Ciraolo and Gatti \cite{CG2025}, which provides an alternative proof of the stability result in our setting. Although their exponent is not sharp, their method appears to have potential for generalization to other contexts, such as the anisotropic setting.

This paper is organized as follows.  We establish a series of vectorial inequalities in Section 2, and prove the corresponding spectral gap estimates Section 3. In the last section, we prove our main result Theorem \ref{main-result}.

\medskip

\noindent {\bf Notation}. In our  estimates, we often express positive constants as $C(\cdot)$ and $c(\cdot)$, with the parentheses enclosing all the parameters upon which the constant depends. Typically, we reserve  $C$ for constants greater than $1$ and $c$ for constants less than  1. When the constant is absolute, we omit the parentheses and simply write $C$ or $c$. It is important to note that the value of  $C(\cdot)$ may differ in various instances, even within a single chain of inequalities.

\section{Sharp Vector inequalities in Euclidean Spaces}

This section aims at establishing sharp vector inequalities, following a similar idea of \cite[Section 2]{FZ2022}: For fixed $x\in \mathbb{R}^n,$ by introducing weights $\omega_j\ (j=1,\cdots,4)$ related to the sizes of both $|x|$ and $|x+y|,$ we construct a quadratic-type expression plus a positive extra term that control $|x+y|^{p-2}(x+y)\cdot x-|x|^{p-2}x\cdot y$ from below, which improves \cite[Lemma 4.2]{ZZ2023}.

\begin{lemma}\label{key-ineq-p-lem}
    Let $x,y\in \mathbb{R}^n$ and $\kappa>0.$  The following inequalities hold.
    \begin{enumerate}
        \item \ For $1<p<2,$ there exists a constant $c_1=c_1(p,\kappa)>0$ such that 
        \begin{align}\label{key-ineq-p-lem-r1}
            |x+y|^{p-2}(x+y)\cdot y\ge &\ |x|^{p-2}x\cdot y+(1-\kappa)\textcolor{blue}{\omega_1}|y|^2+(p-2)(1-\kappa)\textcolor{blue}{\omega_2}\big(|x|-|x+y|)^2\nonumber\\
            &+c_1\min\{|y|^p,|x|^{p-2}|y|^2\}\big). 
        \end{align}
         where
        $$ \omega_1=\omega_1(x,x+y):=
            \left\{
            \begin{array}{ll}
            |x+y|^{p-2} & \text{if}  \  |x|\le|x+y|\\
            |x|^{p-2} & \text{ if }  \  |x+y|\le |x|
            \end{array}
            \right.$$
            and
        $$ \omega_2=\omega_2(x,x+y):=
            \left\{
            \begin{array}{ll}
            \frac{|x+y|^{p-1}}{(2-p)|x+ y|+(p-1)|x|} & \text{if}  \  |x|\le |x+y|\\
            |x|^{p-2} & \text{ if }  \  |x+y|\le |x|
            \end{array}
            \right..$$
        \item \ For $p\ge 2,$  there exist  constants $c_2=c_2(p,\kappa)>0$ and $c_3=c_3(p)>0$ such that
        \begin{align}\label{key-ineq-p-lem-r2}
           |x+y|^{p-2}(x+y)\cdot y\ge &\ |x|^{p-2}x\cdot y+(1-\kappa)\omega_3|y|^2+(p-2)(1-\kappa)\omega_4\big(|x|-|x+y|)^2\nonumber\\
           &+c_2|y|^p.
        \end{align}
        where
        $$ \omega_3=\omega_3(x,x+y):=
            \left\{
            \begin{array}{ll}
            |x|^{p-2}& \text{if}  \  |x|\le |x+y|\\
            \frac{|x+y|^{p-1}}{|x|} & \text{ if }  \  c_3^{\frac{1}{p-1}}|x|\le |x+y|\le |x|\\
           c_3|x|^{p-2}& \text{if} \ |x+y|\le c_3^{\frac{1}{p-1}}|x|
            \end{array}
            \right.$$
            and
        $$ \omega_4=\omega_4(x,x+y):=
            \left\{
            \begin{array}{ll}
            |x|^{p-2}& \text{if}  \  |x|\le |x+y|\\
           \frac{|x+y|^{p-1}}{|x|} & \text{ if }  \  |x+y|\le |x|
           \end{array}
            \right..$$
        \end{enumerate}
        Moreover, for all $x,y \in \mathbb R^n$ it holds \begin{equation}\label{lower-estimate}
            c_3|x|^{p-2}\le \omega_3(x,x+y).
        \end{equation}
\end{lemma}

\begin{proof}
When $|x|=0,$ it is trivial since both $(1)$ and $(2)$ are reduced to $|y|^p\ge 0.$ Subsequently, we assume that $|x|\ne 0$ and divide the proof into two steps.

\noindent\textbf{Step $1:$}  We first construct a quadratic lower bound of
$|x+y|^{p-2}(x+y)\cdot y- |x|^{p-2}x\cdot y$ for $1<p<n.$

It is divided into two cases. 

\noindent\textbf{Case $(1):$ $1<p<2.$} When $|x+y|\le|x|,$ we claim that
\begin{equation}\label{key-ineq-p-lem-s1}
    |x+y|^{p-2}(x+y)\cdot y\ge |x|^{p-2}x\cdot y+|x|^{p-2}|y|^2+(p-2)|x|^{p-2}(|x|-|x+y|)^2.
\end{equation}
Dividing both sides of \eqref{key-ineq-p-lem-s1} by $|x|^p,$ we rearrange the terms and get 
\begin{equation*}\label{key-ineq-p-lem-s1-2}
    \left(\frac{|x+y|^{p-2}}{|x|^{p-2}}-1\right)\left(\frac{x+y}{|x|}\right)\cdot\left(\frac{x+y}{|x|}-\frac{x}{|x|}\right)\ge (p-2)\left(1-\frac{|x+y|}{|x|}\right)^2.
\end{equation*}
Noting
\begin{equation*}
    \frac{x+y}{|x|}\cdot\frac{x}{|x|}\le \frac{|x+y|}{|x|},
\end{equation*}
it is sufficient to prove that 
\begin{equation}\label{key-ineq-p-lem-s1-3}
    \left(\frac{|x+y|^{p-2}}{|x|^{p-2}}-1\right)\left(\frac{|x+y|^2}{|x|^2}\right)- \left(\frac{|x+y|^{p-2}}{|x|^{p-2}}-1\right)\frac{|x+y|}{|x|}- (p-2)\left(1-\frac{|x+y|}
    {|x|}\right)^2\ge 0.
\end{equation}
Let $t=\frac{|x+y|}{|x|},$ and then $0<t\leq 1.$ Observe that the left hand side of of \eqref{key-ineq-p-lem-s1-3} is equivalent to 
$$t^2 (t^{p - 2} - 1) - [(p - 2) (1 - t)^2 + (t^{p - 2} - 1) t]= \frac{(1 - t) ( (2-p) t + ( p-1) t^2 - t^p)} t.$$
As $\frac{1}{p-1}, \ \frac{1}{2-p}> 1$ and $(p-1)+(2-p)=1$, Young's inequality implies
$$t^p = t^{2(p-1) + (2-p)}\le (p-1)t^{2} + (2-p) t. $$
Then, we reach 
$$ (2-p) t + ( p-1) t^2 - t^p\ge 0 \quad \text{ for } 0<t\le 1,$$
and  \eqref{key-ineq-p-lem-s1-3} follows. Thus, we conclude \eqref{key-ineq-p-lem-s1}.

When $|x|\le|x+y|,$ we claim that
\begin{align}\label{key-ineq-p-lem-s1-4}
    |x+y|^{p-2}(x+y)\cdot y\ge &\ |x|^{p-2}x\cdot y+|x+y|^{p-2}|y|^2\nonumber\\
            &+(p-2)\frac{|x+y|}{(2-p)|x+y|+(p-1)|x|}|x+y|^{p-2}\big(|x|-|x+y|)^2.
\end{align}
Towards this, we write $t=\frac{|x+y|}{|x|}\in [1,+\infty)$ and prove that  
\begin{align*}
  f(t)&:=(t-1)\left(t^{p-2}-(p-2)\frac{t^{p-1}}{(2-p)t+(p-1)}(t-1)-1\right)\\
  &=\frac{t-1}{(2-p)t+(p-1)}[(p-1)t^{p-2}-(p-2)t^p+(p-2)t+1-p]\\
  &=: \frac{t-1}{(2-p)t+(p-1)} g(t)\ge 0,\quad \text{for all}\ t\ge 1.
\end{align*}
Note that $\frac{t-1}{(2-p)t+(p-1)}\ge 0,$ and 
then it is enough to show $g(t)\ge 0$ for all $t\ge 1.$
Indeed, a direct computation yields $$g(1)=0\quad \text{and}\quad g'(t)=(p-2)[(p-1)t^{p-3}-pt^{p-1}+1].$$
Furthermore, write $h(t):=(p-1)t^{p-3}-pt^{p-1}+1,$ and then $h(1)=0$ and $$h'(t)=(p-1)t^{p-4}((p-3)-pt^2)< 0,\quad \forall\ t\ge 1,$$ which implies
$h(t)\le 0$ for any $t\ge 1.$ As $p<2,$ this further gives $g'(t)\ge 0.$ Thus, we derive $g(t)\ge 0$ for any $t\ge 1$ and \eqref{key-ineq-p-lem-s1-4} follows. 

\noindent\textbf{Case $(2):$ $p\ge 2.$} For $p\ge 2$ and $|x|\le |x+y|,$ we show that
\begin{equation}\label{key-ineq-p-lem-s2-1}
|x+y|^{p-2}(x+y)\cdot y-|x|^{p-2}(x+y)\cdot y\ge (p-2)|x|^{p-2}(|x|-|x+y|)^2.  
\end{equation}
Note that this inequality holds trivially when $p=2$, and we may assume $p>2$ in what follows.

Dividing both sides of \eqref{key-ineq-p-lem-s2-1} by $|x|^p$ and letting $t=\frac{|x+y|}{|x|}\ge 1,$ it is sufficient to prove 
$$\frac{t^{p}}{p-1}+\frac{(p-2)t}{p-1}-t^2\ge0.$$ This is a direct consequence of Young's inequality since $\frac{p}{p-1}+\frac{p-2}{p-1}=2$ and $p-1> 1,\ \frac{p-1}{p-2}> 1.$ 

When $|x+y|<|x|,$, we first prove the existence of a constant $c_3:=c_3(p)>0$ such that, when $|x+y|\le  c_3^{\frac{1}{p-1}}|x|,$ 
$$|x+y|^{p-2}(x+y)\cdot y\ge |x|^{p-2}x\cdot y+c_3 |x|^{p-2}|y|^2+(p-2)\frac{|x+y|}{|x|}|x+y|^{p-2}(|x|-|x+y|)^2\ge 0.$$
Since $x\cdot (x+y)\le |x||x+y|$, it suffices to show that, there exists a constant $0<c_3=c_3(p)\le \frac{1}{2},$ for which, when $a=c_3(p),$  it holds
$$F(t,a)=(2-p)t^{p+1}+(2p-3)t^p+(1-p)t^{p-1}-(t-1)^2a+(1-t)\ge 0,\quad \forall\  t=\frac{|x+y|}{|x|}\in \left(0, a^{\frac{1}{p-1}}\right].$$
Indeed, when $a=\frac{1}{3},$ we have 
$$\underset{t\rightarrow 0}{\lim}\ F\left(t,\frac{1}{3}\right)=\frac{2}{3}>0. $$
Then since $F$ is  continuous when $t$ is near the origin, there exists a neighborhood $b=b(p)$ of 0, such that $F(t,\frac{1}{3})>0$ for any $0<t<b.$ In particular, by choosing $c_3=c_3(p)>0$ with $0<c_3^{\frac{1}{p-1}}\le \min\{\frac{1}{3},b\},$ $$F\left(t,\frac{1}{3}\right)>0,\ \forall\ t\in \left(0,c_3^{\frac{1}{p-1}}\right).$$
Notice that $F(t,a)$ is monotonically decreasing with respect to $a,$ and then 
$$F\left(t,c_3^\frac{1}{p-1}\right)\ge F\left(t,\frac{1}{3}\right)>0,\ \forall\ t\in \left(0,c_3^{\frac{1}{p-1}}\right).$$

Finally, when $c_3^{\frac{1}{p-1}}|x|\le |x+y|\le |x|,$ we show that
$$|x+y|^{p-2}(x+y)\cdot y\ge  |x|^{p-2}x\cdot y + |x+y|^{p-1}|x|^{-1}[|y|^2+
(p-2)(|x+y|-|x|)^2]. $$
This is equivalent to
\begin{align*}
    |x+y|^{p} + |x|^{p}  \ge & \  \left[|x+y|^{p-2} + |x|^{p-2} -  2 |x+y|^{p-1}|x|^{-1}\right](x+y)\cdot x \\
    &+ |x+y|^{p-1}|x|^{-1}[ (p-1)|x+y|^2 + (p-1)|x|^2-2(p-2)|x+y||x|].
\end{align*}
Observe that
$$|x+y|^{p-2} + |x|^{p-2} -2 |x+y|^{p-1}|x|^{-1} =\frac{|x+y|^{p-1}}{|x+y|}-\frac{|x+y|^{p-1}}{|x|}+\frac{|x|^{p-1}-|x+y|^{p-1}}{|x|}\ge 0.$$
Then it suffices to show that
\begin{align}\label{key-ineq-p-lem-r2-s22}
    |x+y|^{p} + |x|^{p}  \ge & \  \left[|x+y|^{p-2}|x| + |x|^{p-1} -  2 |x+y|^{p-1} \right]|x+y| \nonumber\\
   &+ |x+y|^{p-1}|x|^{-1}[ (p-1)|x+y|^2 + (p-1)|x|^2-2(p-2)|x+y||x|].
\end{align}
To prove this, leting $t=\frac{|x+y|}{|x|},$ it is sufficient to prove $$f(t)=(2p-1)t^p+1-t-(p-1)t^{p+1}-p t^{p-1}= (1-t)[(p-1)t^{p}-pt^{p-1}+1]\ge 0.$$
Note that in this case, $c_3^{\frac{1}{p-1}}\le t\le 1$ and $g(t):=(p-1)t^{p}-pt^{p-1}+1\ge 0$ as $$g(1)=0\quad \text{and}\quad g'(t)=p(p-1)t^{p-2}(t-1)\le 0.$$ This implies that $f(t)=(1-t)g(t)\ge 0.$
Therefore, we conclude \eqref{key-ineq-p-lem-r2-s22} and the desired inequality \eqref{key-ineq-p-lem-s2-1} follows. 

\noindent\textbf{Step $2:$} We again discuss in two cases. For $1<p<2 $ and $|x+y|\le |x|,$ recall $(2.2)$ in \cite[Lemma 2.1]{FZ2022}, 
\begin{equation}\label{nonegative-esti-1}
    |x|^{p-2}|y|^2+(p-2)|x|^{p-2}(|x|-|x+y|)^2\ge c(p)\frac{|x|}{|x|+|y|}|x|^{p-2}|y|^2\ge 0,\quad \text{where}\ c(p)>0.
\end{equation}
    Since $1<p<2$ and $|x|\le |x+y|,$ by the triangle inequality, we have 
\begin{align}\label{nonnegative-esti-2}
    &\quad |x+y|^{p-2}|y|^2+(p-2)\frac{|x+y|}{(2-p)|x+y|+(p-1)|x|}|x+y|^{p-2}(|x|-|x+y|)^2\nonumber\\
    &\ge |x+y|^{p-2}|y|^2+(p-2)\frac{|x+y|}{(2-p)|x+y|+(p-1)|x|}|x+y|^{p-2} |y|^2\nonumber\\
    &= |x+y|^{p-2}|y|^2 \big(\frac{(2-p)|x+y|+(p-1)|x|+(p-2)|x+y|}{(2-p)|x+y|+(p-1)|x|}\big)\nonumber\\
    &= |x+y|^{p-2}|y|^2\frac{(p-1)|x|}{(2-p)|x+y|+|x|}\nonumber\\
    &\ge c(p)\frac{|x|}{|x|+|y|}|x+y|^{p-2}|y|^2\ge 0.
\end{align}
Thus, for any $\kappa >0,$ when $1<p<2,$ it follows from \eqref{key-ineq-p-lem-s1} and \eqref{key-ineq-p-lem-s1-4}, combined with \eqref{nonegative-esti-1} and \eqref{nonnegative-esti-2} respectively, that
\begin{equation*}
     |x+y|^{p-2}(x+y)\cdot y\ge |x|^{p-2}x\cdot y+(1-\kappa)\big(\omega_1|y|^2+(p-2)\omega_2(|x|-|x+y|)^2\big),
\end{equation*}
and the equality holds only when $y=0.$
Similarly, a version of the above inequality also holds for the case when $p\ge 2$, and the equality holds if only if $y=0$ as well.  
Now we can derive inequalities \eqref{key-ineq-p-lem-r1} and \eqref{key-ineq-p-lem-r2} by the arguments similar to Step $(i)$-$3$ and   Step $(ii)$-$2$ in the proof of \cite[Lemma 2.1]{FZ2022}, respectively. 

Finally, we prove \eqref{lower-estimate}. It holds clearly when $|x|\le |x+y|$ and $|x+y|\le c_3^{\frac{1}{p-1}}|x|.$ When $c_3^{\frac{1}{p-1}}|x|\le |x+y|\le |x|,$ it is easy to check since 
    $f(t)=t^{p-1}$ is strictly increasing  when $t>0.$
\end{proof}

The proof of the following lemma is similar to that of \cite[Lemma 2.4]{FZ2022} and \cite[Lemma 3.2]{FN2019}. We postpone the proof with full details in the appendix.

\begin{lemma}\label{key-ineq-p*-lem}
The two following inequalities hold. Let $\kappa>0.$
    \begin{enumerate}
        \item For $1<p\le \frac{2n}{n+2},$ it holds 
        \begin{equation}\label{key-ineq-p*-lem-r1}
        |a+b|^{p^*-2}(a+b)b\le |a|^{p^*-2}ab+ (p^*-1+\kappa)\frac{(|a|+C_1|b|)^{p^*}}{|a|^2+|b|^2}|b|^2,
        \end{equation}
        where $a,b\in \mathbb{R}, a\ne 0$ and $ C_1$ is a positive constant depending on $p^*$ and $\kappa.$
        \item For $\frac{2n}{n+2}<p<\infty,$ it holds 
        \begin{equation}\label{key-ineq-p*-lem-r2}
        |a+b|^{p^*-2}(a+b)b\le|a|^{p^*-2}ab+(p^*-1+\kappa)|a|^{p^*-2}|b|^2+ C_2|b|^{p^*},   
        \end{equation}
        where $a,b\in \mathbb{R}, a\ne 0$ and $C_2$ is a positive constant depending on $p^*$ and $\kappa.$
    \end{enumerate}
\end{lemma}

\section{Spectral gap estimate}

Let $v
\in \mathcal{M}$ and define the linearized $p$-Laplacian operator at $v$ as 
\begin{equation*}
\mathcal{L}_v[\varphi]:=-{\rm div}\left(|Dv|^{p-2}D\varphi+(p-2)|Dv|^{p-4}(Dv\cdot D\varphi)Dv\right).
\end{equation*}
In \cite[Proposition 3.1]{FN2019}, Figalli and Neumayer proved that the manifold 
\begin{equation*}  T_v\mathcal{M}:={\rm span}\{v,\partial_\lambda v,\partial_{z_1}v,\cdots, \partial_{z_n}v\}
\end{equation*}
generates the first and the second eigenspaces of  $\mathcal{L}_v$ 
for $p>2,$ where $z_i$ is the $i$-th element of $z.$ Furthermore, in Section 3.2 of \cite{FZ2022}, Figalli and Zhang generalized this fact to the full range $1<p<n.$
Besides, they utilized the following definition.
\begin{definition}
    For any function $\varphi\in L^{p^*}(\mathbb{R}^n),$ if 
    \begin{equation*}
        \int_{\mathbb{R}^n}v^{p^*-2}\xi\varphi dx=0,\quad \forall \ \xi\in T_v\mathcal{M},
    \end{equation*}
    then we say $\varphi$ is orthogonal to $T_v\mathcal{M}$ in $L^2(\mathbb{R}^n;v^{p^*-2}).$
\end{definition}
\begin{lemma}\label{orth-V-lem}
  There exist $\tilde{\delta}=\tilde{\delta}(n,p)>0$ and a modulus of continuity $\omega:\mathbb{R}_+\rightarrow \mathbb{R}_+$ such that the following statement holds: Let $\|u\|_{L^{p^*}(\mathbb{R}^n)}=1,\ \delta>0$ small enough and 
  $$\|Du-aDU[0,1]\|_{L^p(\mathbb{R}^n)}\le \delta.$$
  If $\delta\le \tilde{\delta},$ then there exists  $v\in \mathcal{M},$ such that $u-v$ is orthogonal to $T_v\mathcal{M}$ and $$\|Du-Dv\|_{L^p(\mathbb{R}^n)}\le \omega(\delta).$$
\end{lemma}
According to Lemma \ref{orth-V-lem}, we write $u= v+\epsilon\varphi,$ such that $\varphi$ is orthogonal to $T_v\mathcal{M},$  $\|D\varphi\|_{L^p(\mathbb{R}^n)}=1$ and $\epsilon\leq \omega(\delta).$

Note that \eqref{norm-value} yields $\|v\|_{L^{p^*}(\mathbb{R}^n)}^{p-p^*}=S^{-p}.$ Moreover, we recall the spectral gap lemma from  \cite[Proposition 3.6]{FZ2022}. 
\begin{lemma}\label{Poincare-ineq}
 For $1<p<n$ and $\varphi\in L^2(\mathbb{R}^n;v^{p^*-2})$ orthogonal to $T_v\mathcal{M},$ there exists a positive constant $\lambda=\lambda(n,p)>0,$ such that
 $$\int_{\mathbb{R}^n}|Dv|^{p-2}|D\varphi|^2+(p-2)|Dv|^{p-4}|Dv\cdot D\varphi|^2 dx \ge \big(p^*-1+2\lambda S^{-p}\big)\int_{\mathbb{R}^n}v^{p^*-2}|\varphi|^2 dx.$$
\end{lemma}

In \cite[Proposition 3.8]{FZ2022}, Figalli and Zhang proved the spectral gap in  Lemma \ref{Poincare-ineq} is stable under a modification of the coefficient, which corresponds to \cite[Lemma 2.1]{FZ2022}. Analogously, we prove the following stability result of the same spectral gap,  which is compatible to our new inequalities Lemma \ref{key-ineq-p-lem}. 
\begin{prop}\label{sepctrum}
 For any $\gamma_0>0$ and $C_1>0,$ there exists $\bar{\delta}=\bar{\delta}(n,p,\gamma_0,C_1)>0$ such that the following statement holds:
 
Let $\varphi\in \dot{W}^{1,p}(\mathbb{R}^n)$ be orthogonal to $T_v\mathcal{M}$ in $L^2(\mathbb{R}^n;v^{p^*-2})$ and satisfy 
 $$\|\varphi\|_{\dot{W}^{1,p}(\mathbb{R}^n)}\le \bar{\delta},$$
 and define $\omega_j=\omega_j(Dv,Dv+D\varphi)$ (namely, choosing $x=Dv$ and $y=D\varphi$ in the definition of $\omega_j$ in Lemma \ref{key-ineq-p-lem}). Then
  \begin{enumerate}
        \item For $1<p\le \frac{2n}{n+2},$ it holds
        \begin{multline}\label{gap-ineq-r1}
            \int_{\mathbb{R}^n}\omega_1|D\varphi|^2+(p-2)\omega_2(|D(v+ \varphi)|-|Dv|)^2dx  
             +\gamma_0\int_{\mathbb{R}^n}\min\{ |D\varphi|^p,\ |Dv|^{p-2}|D\varphi|^2\} dx \\
            \ge \big(p^*-1+\lambda S^{-p}\big)\int_{\mathbb{R}^n}\frac{(v+C_1|\varphi|)^{p^*}}{v^2+|\varphi|^2}|\varphi|^2 dx.
        \end{multline}
            \item For $\frac{2n}{n+2}<p<2,$ it holds 
             \begin{multline}\label{gap-ineq-r2}
            \int_{\mathbb{R}^n}\omega_1|D\varphi|^2+(p-2)\omega_2(|D(v+ \varphi)|-|Dv|)^2dx  
             +\gamma_0\int_{\mathbb{R}^n}\min\{ |D\varphi|^p,\ |Dv|^{p-2}|D\varphi|^2\} dx \\
            \ge \big(p^*-1+\lambda S^{-p}\big)\int_{\mathbb{R}^n}v^{p^*-2}|\varphi|^2 dx.
        \end{multline} 
        \item For $p\ge 2,$ it holds 
        \begin{equation}\label{gap-ineq-r3} \int_{\mathbb{R}^n}\omega_3|D\varphi|^2+(p-2)\omega_4(|D(v+ \varphi)|-|Dv|)^2dx
            \ge \big(p^*-1+\lambda S^{-p}\big)\int_{\mathbb{R}^n}v^{p^*-2}|\varphi|^2 dx.
        \end{equation}      
    \end{enumerate}
    \end{prop}
    
    \begin{proof}
    We prove the result by contradiction in all three cases.
        So, in all cases we shall have a sequence $\varphi_i$, with $\|\varphi_i\|_{\dot{W}^{1,p}(\mathbb{R}^n)}\to 0,$ along which the statement is false. 
        
        \noindent {\bf Case} $(1):$ $1<p<\frac{2n}{n+2}.$ Suppose that the inequality \eqref{gap-ineq-r1} fails. Then there exists a sequence $0\not\equiv \varphi_i\rightarrow 0$ in $\dot{W}^{1,p}(\mathbb{R}^n),$  orthogonal to $T_v\mathcal{M},$ such that 
        \begin{multline*}\label{gap-ineq-r1-s1}
            \int_{\mathbb{R}^n}\omega_{1,i}|D\varphi_i|^2+(p-2)\omega_{2,i}(|D(v+ \varphi_i)|-|Dv|)^2dx  \\
             +\gamma_0\int_{\mathbb{R}^n}\min\{ |D\varphi_i|^p,\ |Dv|^{p-2}|D\varphi_i|^2\} dx \\
            <\big(p^*-1+\lambda S^{-p}\big)\int_{\mathbb{R}^n}\frac{(v+C_1|\varphi_i|)^{p^*}}{v^2+|\varphi_i|^2}|\varphi_i|^2 dx,
        \end{multline*}
        where $\omega_{1,i}$ and $\omega_{2,i}$ corresponds to $\varphi_i$ (i.e., $\omega_{1,i}=\omega_1(Dv,Dv+D\varphi_i)$ and $\omega_{2,i}=\omega_2(Dv,Dv+D\varphi_i)$).

        Define $$\epsilon_i:=\left(\int_{\mathbb{R}^n}(|Dv|+|D\varphi_i|)^{p-2}|D\varphi_i|^2 dx\right)^{\frac{1}{2}}\le \left(\int_{\mathbb{R}^n}|D\varphi_i|^p\right)^\frac{1}{2} dx\rightarrow 0,\quad \text{as}\ i\rightarrow \infty$$
        and $\hat{\varphi}_i:=\frac{\varphi_i}{\epsilon_i}.$ Also, set 
        \begin{equation*}
            \mathcal{R}_i:=\{2|Dv|\ge |D\varphi_i|\},\quad \mathcal{S}_i:=\{2|Dv|<|D\varphi_i|\},
        \end{equation*}
        and for every $R>1,$
        \begin{equation*}
            \mathcal{R}_{i,R}:=(B(0,R)\setminus B(0,1/R))\cap \mathcal{R}_i,\quad \mathcal{S}_{i,R}:=(B(0,R)\setminus B(0,1/R))\cap \mathcal{S}_i.
        \end{equation*}
    Combining with \eqref{nonegative-esti-1} and \eqref{nonnegative-esti-2}, we deduce that 
     \begin{multline}\label{gap-ineq-r1-s2}
           \int_{B(0,R)\setminus B(0,1/R)}\omega_{1,i} |D\hat{\varphi}_i|^2+(p-2)\omega_{2,i}\left(\frac{|D(v+ \varphi_i)|-|Dv|}{\epsilon_i}\right)^2dx  \\
        +\gamma_0\int_{B(0,R)\setminus B(0,1/R)}\min\{\epsilon_i^{p-2} |D\hat{\varphi}_i|^p,\ |Dv|^{p-2}|D\hat{\varphi}_i|^2\} dx \\
            \le \big(p^*-1+\lambda S^{-p}\big)\int_{\mathbb{R}^n}\frac{(v+C_1|\varphi_i|)^{p^*}}{v^2+|\varphi_i|^2}|\hat{\varphi}_i|^2 dx
        \end{multline}
    and 
    \begin{equation*}
        \omega_{1,i}|D\hat{\varphi}_i|^2+(p-2)\omega_{2,i}\left(\frac{|Dv+D\varphi_i|-|Dv|}{\epsilon_i}\right)^2
        \ge c(p)\omega_{1,i}|D\hat{\varphi}_i|^2,\quad \text{on}\ \mathcal{R}_{i,R}.
    \end{equation*}
    As a consequence, 
    \begin{multline}\label{gap-ineq-r1-s3}
c(p)\int_{\mathcal{R}_{i,R}}\omega_{1,i}|D\hat{\varphi}_i|^2dx +\gamma_0\int_{\mathcal{S}_{i,R}}\epsilon_i^{p-2}|D\hat{\varphi}_i|^p dx \\
        \le \int_{B(0,R)\setminus B(0,1/R)}\omega_{1,i}|D\hat{\varphi}_i|^2+(p-2)\omega_{2,i}\left(\frac{|D(v+ \varphi_i)|-|Dv|}{\epsilon_i}\right)^2dx  \\
        +\gamma_0\int_{B(0,R)\setminus B(0,1/R)}\min\{\epsilon_i^{p-2} |D\hat{\varphi}_i|^p,\ |Dv|^{p-2}|D\hat{\varphi}_i|^2\} dx \\
            \le \big(p^*-1+\lambda S^{-p}\big)\int_{\mathbb{R}^n}\frac{(v+C_1|\varphi_i|)^{p^*}}{v^2+|\varphi_i|^2}|\hat{\varphi}_i|^2 dx.
    \end{multline}
    As the triangle inequality tells
    $$|Dv+D\varphi_i|\le |Dv|+|D\varphi_i|\le 3|Dv|,\quad \text{on}\ \mathcal{R}_i,$$
    then 
    \begin{equation}\label{gap-ineq-r1-s7}
       \omega_{1,i}\ge C(p)|Dv|^{p-2},\quad \text{on}\ \mathcal{R}_i. 
    \end{equation}
   Combining \eqref{gap-ineq-r1-s7} with \eqref{gap-ineq-r1-s3}, we get
    \begin{align}\label{gap-ineq-r1-s4}
        1&=\epsilon_i^{-2}\int_{\mathbb{R}^n}(|Dv|+|D\varphi_i|)^{p-2}|D\varphi_i|^2 dx\nonumber\\
         &\le C(p)\left(\int_{\mathcal{R}_i}|Dv|^{p-2}|D\hat{\varphi}_i|^2dx+\int_{\mathcal{S}_i}\epsilon_i^{p-2}|D\hat{\varphi}_i|^p dx\right)\nonumber \\
        &\le C(p)\left(\int_{\mathcal{R}_i}\omega_{1,i}|D\hat{\varphi}_i|^2dx+\int_{\mathcal{S}_i}\epsilon_i^{p-2}|D\hat{\varphi}_i|^p dx\right) \nonumber\\
            &\le C(n,p,\gamma_0)\big(p^*-1+\lambda S^{-p}\big)\int_{\mathbb{R}^n}\frac{(v+C_1|\varphi_i|)^{p^*}}{v^2+|\varphi_i|^2}|\hat{\varphi}_i|^2 dx.
    \end{align}
    Furthermore, referring to \cite[Corollary 3.5]{FZ2022}, when $i$ large enough, we have 
    \begin{equation}\label{gap-ineq-r1-s4.5}
\int_{\mathbb{R}^n}\frac{(v+C_1|\varphi_i|)^{p^*}}{v^2+|\varphi_i|^2}|\hat{\varphi_i}|^2 dx\le  C(n,p,C_1)\int_{\mathbb{R}^n}(|Dv|+|D\varphi_i|)^{p-2}|D\hat{\varphi}_i|^2 dx\le C(n,p,C_1).
    \end{equation}
    Then,  by \eqref{gap-ineq-r1-s3}, \eqref{gap-ineq-r1-s4.5} and the definition of $\mathcal{S}_{i,R},$
    \begin{equation*}
\epsilon_i^{-2}\int_{\mathcal{S}_{i,R}}|Dv|^{p}dx\le \epsilon_i^{p-2}\int_{\mathcal{S}_{i,R}}|D\hat{\varphi}_i|^pdx\le C(n,p,C_1).
    \end{equation*}
 As $|Dv|$ is uniformly bounded away from zero inside $B(0,R)\setminus B(0,1/R),$ it yields
    \begin{equation}\label{gap-ineq-r1-s4.55}
        |\mathcal{S}_{i,R}|\rightarrow 0, \ i\rightarrow \infty,\quad \forall\ R>1.
    \end{equation}
    Now, according to the compactness result \cite[Lemma 3.4]{FZ2022}, \eqref{gap-ineq-r1-s4} and \eqref{gap-ineq-r1-s4.5} imply that there exists $\hat{\varphi}\in \dot{W}^{1,p}(\mathbb{R}^n)\cap L^2(\mathbb{R}^n,v^{p^*-2})$ such that $$\hat{\varphi}_i\rightharpoonup \hat{\varphi},\ \text{in} \ \dot{W}^{1,p}(\mathbb{R}^n),\quad \text{as}\ i\rightarrow \infty,$$
    and 
    \begin{equation}\label{gap-ineq-r1-s5}
 \int_{\mathbb{R}^n}\frac{(v+C_1|\varphi_i|)^{p^*}}{v^2+|\varphi_i|^2}|\hat{\varphi_i}|^2 dx\rightarrow \int_{\mathbb{R}^n}v^{p^*-2}|\hat{\varphi}|^2dx.
    \end{equation}
    Again, by \eqref{gap-ineq-r1-s3} and \eqref{gap-ineq-r1-s4.5}, we have 
    \begin{equation*}
    \int_{\mathcal{R}_{i,R}}\omega_{1,i}|D\hat{\varphi}_i|^2dx \le C(n,p,C_1).
    \end{equation*}
    Combining this with \eqref{gap-ineq-r1-s7}, 
    we derive $$C(p)\int_{\mathcal{R}_{i,R}}|Dv|^{p-2}|D\hat{\varphi}_i|^2 dx\le \int_{\mathcal{R}_{i,R}}\omega_{1,i}|D\hat{\varphi}_i|^2 dx\le C(n,p,C_1),$$
    which furthermore implies that, up to passing to a subsequence, 
    $$D\hat{\varphi}_i \chi_{\mathcal{R}_{i,R}}\rightharpoonup D\hat{\varphi}\chi_{B(0,R)\setminus B(0,1/R)}\quad \text{in}\  
 L^2(\mathbb{R}^n,\mathbb{R}^n),\quad \forall\ R>1,$$
 where we applied \eqref{gap-ineq-r1-s4.55} and the weak convergence of $\hat{\varphi}_i$ to $\hat{\varphi}$ in $\dot{W}^{1,p}(\mathbb{R}^n)$.
 In particular, $\hat{\varphi}\in \dot{W}_{\rm{loc}}^{1,2}(\mathbb{R}^n\setminus\{0\}),$ by the low semi-continuity of the weak limit. Moreover, letting $i\rightarrow 0,$ \eqref{gap-ineq-r1-s4}, \eqref{gap-ineq-r1-s4.5} and \eqref{gap-ineq-r1-s5} yield
 $$0<c(n,p,\gamma_0)\le \|\hat{\varphi}\|_{L^2(\mathbb{R}^n;v^{p^*-2})}\le C(n,p,C_1).$$
 Write $\hat{\varphi}_i=\hat{\varphi}+\psi_i,$ then
  $$\psi_i\rightharpoonup 0\quad \text{in}\ \dot{W}^{1,p}(\mathbb{R}^n)\quad \text{and}\quad D\psi_i\chi_{\mathcal{R}_i}\rightharpoonup 0\quad \text{in}\ L_{\rm{loc}}^2(\mathbb{R}^n\setminus \{0\},\mathbb{R}^n).$$ Set
 $$f_{i,1}:=\left[\int_{0}^1 \frac{Dv+tD\varphi_i}{|Dv+tD\varphi_i|}dt\right]\cdot D\hat{\varphi},\quad f_{i,2}:=\left[\int_{0}^1 \frac{Dv+tD\varphi_i}{|Dv+tD\varphi_i|}dt\right]\cdot D\psi_i,$$
 then 
 \begin{equation*}
    \left(\frac{|Dv+D\varphi_i|-|Dv|}{\epsilon_i}\right)^2=\left(\left[\int_0^1 \frac{Dv+tD\varphi_i}{|Dv+tD\varphi_i|}dt\right]\cdot D\hat{\varphi}_i\right)^2=(f_{i,1}+f_{i,2})^2.
 \end{equation*}
 Exploiting the strong $\dot{W}^{1,p}$-convergence of $\varphi_i$ to 0 and Lebesgue's dominated convergence theorem, we have, up to passing to a subsequence, $|\omega_{1,i}|\rightarrow |Dv|$ a.e.,  
  $$f_{i,1}\rightarrow \frac{Dv}{|Dv|}\cdot D\hat{\varphi}\quad \text{in}\ L_{\rm{loc}}^2(\mathbb{R}^n\setminus \{0\}),\quad  f_{i,2}\chi_{\mathcal{R}_i}\rightharpoonup 0\quad \text{in}\ L_{\rm{loc}}^2(\mathbb{R}^n\setminus \{0\}).$$
 Since $f_{i,2}^2\le |D\psi_i|^2$, \eqref{nonegative-esti-1} and \eqref{nonnegative-esti-2} hold, we have
 \begin{multline}
     \int_{\mathcal{R}_{i,R}}\omega_{1,i}|D\hat{\varphi}_i|^2+(p-2)\omega_{2,i}\left(\frac{|Dv+D\varphi_i|-|Dv|}{\epsilon_i}\right)^2dx\\
     =\int_{\mathcal{R}_{i,R}}\omega_{1,i}\left(|D\hat{\varphi}|^2+2D\psi_i\cdot D\hat{\varphi}\right)+(p-2)\omega_{2,i}(f_{i,1}^2+2f_{i,1}f_{i,2})dx\\
     +\int_{\mathcal{R}_{i,R}}\omega_{1,i}|D\psi_i|^2+(p-2)\omega_{2,i}f_{i,2}^2dx \\
     \ge \int_{\mathcal{R}_{i,R}}\omega_{1,i}\left(|D\hat{\varphi}|^2+2D\psi_i\cdot D\hat{\varphi}\right)+(p-2)\omega_{2,i}(f_{i,1}^2+2f_{i,1}f_{i,2})dx.
 \end{multline}
 By the definition of $\omega_{1,i}$ and $\omega_{2,i},$ 
 \begin{equation*}
     \omega_{1,i}\le C(p)|Dv|^{p-2}\quad \text{and}\quad \omega_{2,i}\le C(p)|Dv|^{p-2}\quad \text{on}\  \mathcal{R}_{i,R}.
 \end{equation*}
 Moreover, employing the following convergences
 \begin{equation*}
     D\psi_i\chi_{\mathcal{R}_i}\rightharpoonup 0,\quad f_{i,1}\rightarrow \frac{Dv}{|Dv|}\cdot D\hat{\varphi},\quad f_{i,2}\chi_{\mathcal{R}_i}\rightharpoonup 0,\quad \text{in}\ L_{\rm{loc}}^2(\mathbb{R}^n\setminus \{0\}),
 \end{equation*}
 \begin{equation*}
     \omega_{1,i}\rightarrow |Dv|^{p-2}\ {\rm a.e.},\quad 
     \omega_{2,i}\rightarrow |Dv|^{p-2}\ {\rm a.e.},\quad
     |(B(0,R)\setminus B(0,1/R))\setminus \mathcal{R}_{i,R}|\rightarrow 0
 \end{equation*}
 with the Lebesgue's dominated convergence theorem, we arrive at
\begin{multline*}
    \underset{i\rightarrow \infty}{\lim}\int_{\mathcal{R}_{i,R}}\omega_{1,i}\left(|D\hat{\varphi}|^2+2D\psi_i\cdot D\hat{\varphi}\right)+(p-2)\omega_{2,i}(f_{i,1}^2+2f_{i,1}f_{i,2})dx\\
    = \int_{B(0,R)\setminus B(0,1/R)}|Dv|^{p-2}|D\hat{\varphi}|^2+(p-2)|Dv|^{p-2}\left(\frac{Dv}{|Dv|}\cdot D\hat{\varphi}\right)^2 dx.
\end{multline*}
This together with \eqref{gap-ineq-r1-s2} and \eqref{gap-ineq-r1-s5} yields, when $R\rightarrow +\infty,$
\begin{equation}\label{gap-ineq-s123}
    \int_{\mathbb{R}^n}|Dv|^{p-2}|D\hat{\varphi}|^2+(p-2)|Dv|^{p-2}\left(\frac{Dv}{|Dv|}\cdot D\hat{\varphi}\right)^2 dx\\
    \le  (p^*-1+\lambda S^{-p})\int_{\mathbb{R}^n}v^{p^*-2}|\hat{\varphi}|^2dx.
\end{equation}
On the other hand, by Sobolev inequality and the orthogonality of $\varphi_i,$ we know $\hat{\varphi}_i\rightharpoonup \hat{\varphi}$ in $L^{p^*}(\mathbb{R}^n)$ and $\hat{\varphi}$ is orthogonal to $T_v\mathcal{M}.$ Notice that $\hat{\varphi}\in L^2(\mathbb{R}^n;v^{p^*-2}),$ then \eqref{gap-ineq-s123} contradicts Lemma \ref{Poincare-ineq}, and we conclude $(1)$ of this lemma.

\medskip

\noindent {\bf Case} $(2):$ $\frac{2n}{n+2}<p<2.$ Assume that \eqref{gap-ineq-r2} fails. Then there exists a sequence $0\not\equiv \varphi_i\rightarrow 0$ in $\dot{W}^{1,p}(\mathbb{R}^n),$ with $\varphi_i$ orthogonal to $T_v\mathcal{M},$ such that 
\begin{multline}\label{gap-ineq-r2-s1}
            \int_{\mathbb{R}^n}\omega_{1,i}|D\varphi_i|^2+(p-2)\omega_{2,i}(|D(v+ \varphi_i)|-|Dv|)^2dx \\ 
             +\gamma_0\int_{\mathbb{R}^n}\min\{ |D\varphi_i|^p,\ |Dv|^{p-2}|D\varphi_i|^2\} dx \\
<\big(p^*-1+\lambda S^{-p}\big)\int_{\mathbb{R}^n}v^{p^*-2}|\varphi_i|^2 dx,
            \end{multline}
            where $\omega_{1,i}$ and $\omega_{2,i}$ corresponds to $\varphi_i.$ 

            Similar to the case $(1),$ we define $$\epsilon_i:=\left(\int_{\mathbb{R}^n}(|Dv|+|D\varphi_i|)^{p-2}|D\varphi_i|^2 dx\right)^{\frac{1}{2}}\rightarrow 0,\quad \text{as}\ i\rightarrow 0,\quad \hat{\varphi}_i:=\frac{\varphi_i}{\epsilon_i}$$
       and split $B(0,R)\setminus B(0,1/R)=\mathcal{R}_{i,R}\cup \mathcal{S}_{i,R}.$
       It follows that \eqref{gap-ineq-r1-s3} and \eqref{gap-ineq-r1-s4} hold by replacing the term in the right-hand side to be $\int_{\mathbb{R}^n} v^{p^*-2}|\hat{\varphi}_i|^2 dx.$ By H\"{o}lder inequality, we reach
       \begin{align*}
       \int_{\mathbb{R}^n}|D\hat{\varphi}_i|^p dx &\le \left(\int_{\mathbb{R}^n}(|Dv|+|D\varphi_i|)^{p-2}|D\hat{\varphi}_i|^2dx\right)^\frac{p}{2}\left(\int_{\mathbb{R}^n}(|Dv|+|D\varphi_i|)^p dx\right)^{1-\frac{p}{2}}\\
       &\le C(p)\left[\left(\int_{\mathbb{R}^n}|Dv|^p dx\right)^{1-\frac{p}{2}}+\epsilon_i^\frac{p(2-p)}{2}\left(\int_{\mathbb{R}^n}|D\hat{\varphi}_i|^p dx\right)^{1-\frac{p}{2}}\right],\\
       &\le C(n,p).
       \end{align*}
       It follows that, up to passing to a subsequence, $$\hat{\varphi}_i\rightharpoonup \hat{\varphi}\quad \text{in}\ \dot{W}^{1,p}(\mathbb{R}^n),\quad \hat{\varphi}_i\rightarrow \hat{\varphi}\quad \text{in}\ L_{\rm{loc}}^2(\mathbb{R}^n).$$
       Moreover, by H\"{o}lder inequality and Sobolev inequality, we have 
       \begin{align*}
           \int_{\mathbb{R}^n\setminus B(0,\rho)}v^{p^*-2}|\hat{\varphi}_i|^2 dx&\le \left(\int_{\mathbb{R}^n\setminus B(0,\rho)}v^{p^*}dx\right)^{1-\frac{2}{p^*}}\left(\int_{\mathbb{R}^n\setminus B(0,\rho)}|\hat{\varphi}_i|^{p^*}dx\right)^\frac{2}{p^*}\\
           &\le \left(\int_{\mathbb{R}^n\setminus B(0,\rho)}v^{p^*}dx\right)^{1-\frac{2}{p^*}}\left(\int_{\mathbb{R}^n}|D\hat{\varphi}_i|^p dx\right)^\frac{2}{p}
           \\
           &\le C(n,p)\left(\int_{\mathbb{R}^n\setminus B(0,\rho)}v^{p^*}dx\right)^{1-\frac{2}{p^*}}\rightarrow 0,\quad \text{as}\ \rho\rightarrow\ \infty.
       \end{align*}
       Combining with the strong $L_{\rm{loc}}^2$-convergence of 
       $\hat{\varphi}_i,$
       we derive
       \begin{equation}\label{orth}
           \hat{\varphi}_i\rightarrow \hat{\varphi},\quad \text{in}\ L^2(\mathbb{R}^n;v^{p^*-2}).
           \end{equation}
       Letting $i\rightarrow \infty,$ by the analog of \eqref{gap-ineq-r1-s3}, we know that $\|\hat{\varphi}\|_{L^2(\mathbb{R}^n;v^{p^*-2})}$ is strictly positive and by that of \eqref{gap-ineq-r1-s4}, we have 
       $$|\mathcal{S}_{i,R}|\rightarrow 0\quad \text{and}\quad  \int_{\mathcal{R}_{i,R}}\omega_{1,i}|D\hat{\varphi}_i|^2 dx\le C(n,p),\quad \forall\ R>1.$$
       Then, combining with  \eqref{gap-ineq-r1-s7}, up to a subsequence, it holds 
     $$D\hat{\varphi}_i \chi_{\mathcal{R}_{i,R}}\rightharpoonup D\hat{\varphi}\chi_{B(0,R)\setminus B(0,1/R)}\quad \text{in}\ 
 L^2(\mathbb{R}^n,\mathbb{R}^n),\quad \forall\ R>1,$$
       since $\hat{\varphi}_i$ is weakly converges to $\hat{\varphi}$ in $\dot{W}^{1,p}(\mathbb{R}^n).$
Set $\hat{\varphi}_i=\hat{\varphi}+\psi_i$ and define $f_{i,1},\ f_{i,2}$ as in the case $p\le \frac{2n}{n+2},$ following the same procedure as in Case $(1)$, we obtain
\begin{multline*}
    \underset{i\rightarrow \infty}{\liminf}\int_{\mathcal{R}_{i,R}}\omega_{1,i}|D\hat{\varphi}_i|^2+(p-2)\omega_{2,i}\left(\frac{|Dv+D\varphi_i|-|Dv|}{\epsilon_i}\right)^2dx\\
    \ge \int_{B(0,R)\setminus B(0,1/R)}|Dv|^{p-2}|D\hat{\varphi}|^2+(p-2)|Dv|^{p-2}\left(\frac{Dv}{|Dv|}\cdot D\hat{\varphi}\right)^2 dx.
\end{multline*}
Combining with \eqref{gap-ineq-r2-s1} and letting $R\rightarrow +\infty,$ we derive  \eqref{gap-ineq-s123}, which contradicts Lemma \ref{Poincare-ineq} as $\hat{\varphi}$ is orthogonal to $T_v\mathcal{M}$ by \eqref{orth}. This concludes $(2)$ of this lemma.

\medskip

\noindent {\bf Case} $(3):$ $p\ge 2.$ Suppose that the inequality \eqref{gap-ineq-r3} fails. Then there exists a sequence $0\not\equiv \varphi_i\rightarrow 0$ in $\dot{W}^{1,p}(\mathbb{R}^n),$ with $\varphi$ orthogonal to $T_v\mathcal{M},$ such that
\begin{multline}\label{gap-ineq-r3-s1}
            \int_{\mathbb{R}^n}\omega_{3,i}|D\varphi_i|^2+(p-2)\omega_{4,i}(|D(v+ \varphi_i)|-|Dv|)^2dx \\
            <\big(p^*-1+\lambda S^{-p}\big)\int_{\mathbb{R}^n}v^{p^*-2}|\varphi_i|^2 dx,
    \end{multline}
    where $\omega_{3,i}$ and $\omega_{4,i}$ corresponds to $\varphi_i.$ 
    
    Write $$\epsilon_i:=\|\varphi_i\|_{\dot{W}^{1,2}(\mathbb{R}^n;|Dv|^{p-2})}\le \left(\int_{\mathbb{R}^n}|Dv|^p dx\right)^\frac{p-2}{p}\left(\int_{\mathbb{R}^n}|D\varphi_i|^p dx\right)^{\frac{2}{p}}\rightarrow 0$$
and $ \hat{\varphi}_i=\frac{\varphi_i}{\epsilon_i},$ then $\|\hat{\varphi}_i\|_{\dot{W}^{1,2}(\mathbb{R}^n;|Dv|^{p-2})}=1.$ By \cite[Proposition 3.2]{FZ2022}, up to passing to a subsequence, we have $\hat{\varphi}_i\rightharpoonup\hat{\varphi}$ in $\dot{W}_{\rm{loc}}^{1,2}(\mathbb{R}^n;|Dv|^{p-2})$  and $\hat{\varphi}_i\rightarrow\hat{\varphi}$ in $L^2(\mathbb{R}^n;v^{p^*-2}).$ 
Combining with \eqref{lower-estimate} and \eqref{gap-ineq-r3-s1}, we have 
$$0<c_3=c_3\int_{\mathbb{R}^n}|Dv|^{p-2}|D\hat{\varphi}_i|^2 dx\le \int_{\mathbb{R}^n}\omega_{3,i}|D\varphi_i|^2\le (p^*-1+\lambda S^{-p})\int_{\mathbb{R}^n} v^{p^*-2}|\hat{\varphi}_i|^2 dx,$$
which implies 
$$\|\hat{\varphi}\|_{L^2(\mathbb{R}^n;v^{p^*-2})}\ge c(n,p)>0.$$
Moreover, since $p\ge 2,$ for any $R>1,$ we get
\begin{multline}\label{gap-ineq-r3-s2}
            \int_{B(0,R)\setminus B(0,1/R)}\omega_{3,i}|D\hat{\varphi}_i|^2+(p-2)\omega_{4,i}\left(\frac{|D(v+\varphi_i)|-|Dv|}{\epsilon_i}\right)^2dx \\
\le\big(p^*-1+\lambda S^{-p}\big)\int_{\mathbb{R}^n}v^{p^*-2}|\hat{\varphi_i}|^2 dx.
    \end{multline}
    Observe that \eqref{lower-estimate} gives   $$0<c(R)\le c_3|Dv|^{p-2}\le \omega_{3,i}\le |Dv|^{p-2}\le C(R),\quad \text{on}\ B(0,R)\setminus B(0,1/R).$$ Let $$\psi_i=\hat{\varphi}_i-\hat{\varphi},$$ then $$\psi_i\rightharpoonup 0,\ \text{in}\ \dot{W}^{1,2}_{\rm{loc}}(\mathbb{R}^n\setminus \{0\}).$$
    Notice that $$\omega_{j,i}\le |Dv|^{p-2}\quad \text{and}\quad \omega_{j,i}\rightarrow |Dv|^{p-2}\ {\rm a.e.},\quad j=\{3,4\}, $$
    then as in the case $\frac{2n}{n-2}<p<2,$ we get
  \begin{multline*}
           \underset{i\rightarrow \infty}{\liminf}\int_{B(0,R)\setminus B(0,1/R)}\omega_{3,i}|D\hat{\varphi}_i|^2+(p-2)\omega_{4,i}\left(\frac{|D(v+\varphi_i)|-|Dv|}{\epsilon_i}\right)^2dx \\
\ge   \int_{B(0,R)\setminus B(0,1/R)}|Dv|^{p-2}|D\hat{\varphi}|^2+(p-2)|Dv|^{p-2}\left(\frac{Dv}{|Dv|}\cdot D\hat{\varphi}\right)^2 dx.
\end{multline*}
Combining with \eqref{gap-ineq-r3-s2} and letting $R\rightarrow +\infty$, we derive \eqref{gap-ineq-s123}, which contradicts Lemma \ref{Poincare-ineq}. This concludes $(3)$ of this lemma.
\end{proof}

\section{Proof of Theorem \ref{main-result}}
With the necessary preparations in place, we are now in a position to prove Theorem \ref{main-result}.

\begin{proof} 
Thanks to Lemma \ref{orth-V-lem}, there exists $v\in \mathcal{M},$ such that $u-v$ is orthogonal to $T_v\mathcal{M}$ and $\|Du-Dv\|_{L^p(\mathbb{R}^n)}\le \omega(\delta).$ Moreover, we write $u=v+\epsilon \varphi$ and assume that  $$\|D\varphi\|_{L^p(\mathbb{R}^n)}=1,\ \epsilon\le \omega(\delta).$$
Recall that $P(u)=-{\rm div}(|Du|^{p-2}Du)-|u|^{p^*-2}u.$ Then testing $P(u)$ against $\epsilon\varphi$ yields 
\begin{equation*}
\left<P(u),\epsilon\varphi\right>=\epsilon\int_{\mathbb{R}^n}-{\rm div}(|Du|^{p-2}Du)\varphi dx -\epsilon \int_{\mathbb{R}^n} |u|^{p^*-2}u \varphi dx.
\end{equation*}
Combining with $$ \left<P(u),\epsilon\varphi\right>\ \le \epsilon \|P(u)\|_{W^{-1,q}(\mathbb{R}^n)} \|D\varphi\|_{L^p(\mathbb{R}^n)},$$ it holds that
\begin{equation}\label{main-ineq}
  \epsilon \int_{\mathbb{R}^n} |Du|^{p-2}DuD\varphi dx - \epsilon\int_{\mathbb{R}^n} |u|^{p^*-2}u\varphi dx\le  \epsilon \|P(u)\|_{W^{-1,q}(\mathbb{R}^n)} \|D\varphi\|_{L^p(\mathbb{R}^n)}. 
\end{equation}

We split the proof into three cases. 

\noindent\textbf{Case $(1)$:} $1<p\le \frac{2n}{n+2}.$ We show there exists $C=C(n.p)>0,$ such that 
        $$\|Du-Dv\|_{L^p(\mathbb{R}^n)}\le C\|P(u)\|_{W^{-1,q}(\mathbb{R}^n)}.$$
        By setting $x=Dv$ and $y=\epsilon D\varphi$ in \eqref{key-ineq-p-lem-r1}, we derive that
    \begin{align}\label{main-r1-s1-1}
        \epsilon \int_{\mathbb{R}^n}|Du|^{p-2}Du\cdot D\varphi &\ge  \epsilon \int_{\mathbb{R}^n}|Dv|^{p-2}Dv\cdot D\varphi dx+\epsilon^2(1-\kappa)\int_{\mathbb{R}^n}\omega_1|D\varphi|^2 dx\nonumber \\
        &\quad+(p-2)(1-\kappa)\int_{\mathbb{R}^n}\omega_2(|Du|-|Dv|)^2dx\nonumber\\
        &\quad +c_1\int_{\mathbb{R}^n}\min\{\epsilon^p|D\varphi|^p,\epsilon^2|Dv|^{p-2}|D\varphi|^2\}dx.
    \end{align}
    Let $a=v$ and $b=\epsilon \varphi$ in \eqref{key-ineq-p*-lem-r1}. It follows that 
    \begin{equation}\label{main-r1-s1-2}
       \epsilon \int_{\mathbb{R}^n} |u|^{p^*-2}u\varphi dx\le \ \epsilon \int_{\mathbb{R}^n}|v|^{p^*-2}v \varphi dx +\epsilon^2(p^*-1+\kappa)\int_{\mathbb{R}^n} \frac{(v+C_1|\epsilon \varphi|)^{p^*}}{v^2+|\epsilon \varphi|^2} \varphi^2dx.
    \end{equation}
    By \eqref{EL-equ} and the orthogonality, it holds that
    \begin{equation}\label{main-r1-s1-3}
    \int_{\mathbb{R}^n} |Dv|^{p-2}Dv\cdot D\varphi dx =\int_{\mathbb{R}^n}v^{p^*-1}\varphi dx=0. 
    \end{equation}
    Combining \eqref{main-ineq}-\eqref{main-r1-s1-3}, we arrive at
    \begin{align*}
        &\quad \epsilon \|P(u)\|_{W^{-1,q}(\mathbb{R}^n)}\|D\varphi\|_{L^p(\mathbb{R}^n)}\\
        & \ge \epsilon^2(1-\kappa)\bigg[ \int_{\mathbb{R}^n}\omega_1|D\varphi|^2+(p-2)\omega_2\left(\frac{|Dv+\epsilon D\varphi|-|Dv|}{\epsilon}\right)^2dx\bigg]\nonumber\\
&\quad+c_1\int_{\mathbb{R}^n}\min\{\epsilon^p|D\varphi|^p,\epsilon^2|Dv|^{p-2}|D\varphi|^2\}dx\nonumber \\
        &\quad -\epsilon^2 (p^*-1+\kappa)\int_{\mathbb{R}^n}\frac{(v+C_1|\epsilon \varphi|)^{p^*}}{v^2+|\epsilon \varphi|^2} \varphi^2 dx.
    \end{align*}
    Then, owing to \eqref{gap-ineq-r1}, we deduce that 
    \begin{align*}
         &\quad \epsilon \|P(u)\|_{W^{-1,q}(\mathbb{R}^n)}\|D\varphi\|_{L^p(\mathbb{R}^n)}\\
        &\ge \epsilon^2\left(1-\kappa-\frac{p^*-1+\kappa}{p^*-1+\lambda S^{-p}}\right)\int_{\mathbb{R}^n}\omega_1|D\varphi|^2\\
        &\quad +(p-2)\left(1-\kappa-\frac{p^*-1+\kappa}{p^*-1+\lambda S^{-p}}\right)\int_{\mathbb{R}^n}\omega_2(|Du|-|Dv|)^2dx\big]\\
        &\quad +\left(c_1-\frac{\gamma_0(p^*-1+\kappa)}{p^*-1+\lambda S^{-p}}\right) \int_{\mathbb{R}^n}\min\{\epsilon^p|D\varphi|^p,\epsilon^2|Dv|^{p-2}|D\varphi|^2\}dx.
    \end{align*}
    Via choosing $\kappa=\kappa(n,p)>0$ small enough such that  
    $$1-\kappa-\frac{p^*-1+\kappa}{p^*-1+\lambda S^{-p}}\ge 0,$$
    and $\gamma_0=\gamma_0(n,p)>0$ small enough such that
    $$\frac{c_1}{2} \le c_1-\frac{(p^*-1+\kappa) \gamma_0}{(p^*-1)+\lambda S^{-p}},$$
    we get
    \begin{equation*}
       \epsilon \|P(u)\|_{W^{-1,q}(\mathbb{R}^n)}\|D\varphi\|_{L^p(\mathbb{R}^n)} \ge \frac{c_1}{2}\int_{\mathbb{R}^n}\min\{\epsilon^p|D\varphi|^p,\epsilon^2|Dv|^{p-2}|D\varphi|^2\}dx.
    \end{equation*}
    On the other hand, according to $(4.5)$ in the proof of \cite[Theorem 1.1]{FZ2022}, it holds 
   \begin{equation}\label{compare-Dvarphi-ineq}
       \int_{\mathbb{R}^n}\min\{\epsilon^p|D\varphi|^p,\epsilon^2|Dv|^{p-2}|D\varphi|^2\}dx\ge c\epsilon^2\|D\varphi\|_{L^p(\mathbb{R}^n)}^2,
   \end{equation}
   where $c=c(n,p)>0.$
  Thus, we conclude that 
   \begin{equation*}
\|Du-Dv\|_{L^p(\mathbb{R}^n)}=\epsilon\|D\varphi\|_{L^p(\mathbb{R}^n)}\le C\|P(u)\|_{W^{-1,q}(\mathbb{R}^n)}.
   \end{equation*}
  
\noindent\textbf{Case $(2):$} $\frac{2n}{n+2}< p< 2.$ We show there exists $C=C(n,p)>0$ such that
        $$\|Du-Dv\|_{L^p(\mathbb{R}^n)}\le C \|P(u)\|_{W^{-1,q}(\mathbb{R}^n)}.$$
        The proof is quite analogous to that of the case $1<p\le \frac{2n}{n+2}$ with a small modification: setting $a=v$ and $b=\epsilon \varphi$ in \eqref{key-ineq-p*-lem-r2},
    \begin{align}\label{main-r2-s1-1}
      \epsilon \int_{\mathbb{R}^n} |u|^{p^*-2}u\varphi dx
       &\le \epsilon \int_{\mathbb{R}^n}|v|^{p^*-2}v\varphi dx +\epsilon^2(p^*-1+\kappa)\int_{\mathbb{R}^n} |v|^{p^*-2} \varphi^2dx\nonumber\\
       &\quad+C_2\epsilon^{p^*}\int_{\mathbb{R}^n}|\varphi|^{p^*}dx .
    \end{align}
    Combining with \eqref{main-r1-s1-1}, \eqref{main-r1-s1-3} and \eqref{gap-ineq-r2}, we get 
     \begin{align*}
         &\quad \epsilon \|P(u)\|_{W^{-1,q}(\mathbb{R}^n)}\|D\varphi\|_{L^p(\mathbb{R}^n)}\\
         &\ge \epsilon^2(1-\kappa) \int_{\mathbb{R}^n}\omega_1|D\varphi|^2dx
         +(p-2)(1-\kappa)\int_{\mathbb{R}^n}\omega_2(|Dv+\epsilon D\varphi|-|Dv|)^2dx\nonumber\\
    &\quad +c_1\int_{\mathbb{R}^n}\min\{\epsilon^p|D\varphi|^p,\epsilon^2|Dv|^{p-2}|D\varphi|^2\}dx \\
        &\quad -\epsilon^2 (p^*-1+\kappa)\int_{\mathbb{R}^n}v^{p^*-2}|\varphi|^2 dx-C_2\epsilon^{p^*}\int_{\mathbb{R}^n}|\varphi|^{p^*}dx\\
        &\ge c\epsilon^2\|D\varphi\|_{L^p(\mathbb{R}^n)}^2-C_2\epsilon^{p^*}\int_{\mathbb{R}^n}|\varphi|^{p^*}dx,
    \end{align*}
    where the last step follows from \eqref{compare-Dvarphi-ineq} and choosing suitable $\kappa, \gamma_0>0.$ Noting that $$1=\|D\varphi\|_{L^p(\mathbb{R}^n)}\ge S\|\varphi\|_{L^{p^*}(\mathbb{R}^n)}$$ and $p^*>2,$ we can derive the desired result provided $\epsilon$ is small enough. 
    
\noindent\textbf{Case $(3):$} $p\ge 2.$ We show there exists $C=C(n,p)>0,$ such that 
        $$\|Du-Dv\|^{p-1}_{L^p(\mathbb{R}^n)}=\|\epsilon D \varphi\|_{L^p(\mathbb{R}^n)}^{p-1}\le C \|P(u)\|_{W^{-1,q}(\mathbb{R}^n)}.$$ Letting $x=Dv, y=\epsilon D\varphi$ in \eqref{key-ineq-p-lem-r2} and $a=v, b=\epsilon \varphi$ in \eqref{key-ineq-p*-lem-r2}, we get \eqref{main-r2-s1-1} and 
        \begin{align}\label{main-r3-s1-1}
           \epsilon\int_{\mathbb{R}^n}|Du|^{p-2}Du\cdot D\varphi dx \ge &\ \epsilon\int_{\mathbb{R}^n} |Dv|^{p-2}Dv\cdot D\varphi dx +(1-\kappa)\epsilon^2\int_{\mathbb{R}^n}\omega_3|D\varphi|^2dx \nonumber\\
           &+(p-2)(1-\kappa)\int_{\mathbb{R}^n}\omega_4\big(|Dv|-|Du|)^2 dx \nonumber\\
           &+c_2\epsilon^p\int_{\mathbb{R}^n} |D\varphi|^p dx.
        \end{align}
         Combining this with the spectral gap inequality \eqref{gap-ineq-r3} and \eqref{main-ineq}, it follows from \eqref{main-r2-s1-1} and \eqref{main-r3-s1-1} that 
    \begin{align*}
        &\quad \epsilon \|P(u)\|_{W^{-1,q}(\mathbb{R}^n)}\|D\varphi\|_{L^p(\mathbb{R}^n)}\nonumber\\
        &\ge \epsilon^2\left(1-\kappa-\frac{p^*-1+\kappa}{p^*-1+\lambda S^{-p}}\right)\left( \int_{\mathbb{R}^n}\omega_3|D\varphi|^2\nonumber
        +(p-2)\omega_4\left(\frac{|Du|-|Dv|}{\epsilon}\right)^2dx\right)\\
        &\quad+c_2\epsilon^p\|D\varphi\|_{L^p}^p-C_2\epsilon^{p^*}\int_{\mathbb{R}^n}|\varphi|^{p^*}dx\\
        &\ge C \epsilon^p\|D\varphi\|_{L^p(\mathbb{R}^n)}^p,
    \end{align*}
    where $C=C(n,p)$ and the last inequality follows from  $1=\|D\varphi\|_{L^p(\mathbb{R}^n)}\ge S \|\varphi\|_{L^{p^*}(\mathbb{R}^n)}$ and $p^*>2.$ Hence, it holds $$\|Du-Dv\|^{p-1}_{L^p(\mathbb{R}^n)}=\epsilon^{p-1}\|D\varphi\|^{p-1}_{L^p(\mathbb{R}^n)}\le C \|P(u)\|_{W^{-1,q}(\mathbb{R}^n)}.$$

    All in all, we conclude Theorem \ref{main-result}.
\end{proof}
\appendix
\section{Proof of Lemma \ref{key-ineq-p*-lem}}
\begin{proof}[Proof of Lemma \ref{key-ineq-p*-lem}]
We divide the proof into two steps.

\noindent\textbf{Step 1:}
We first show  \eqref{key-ineq-p*-lem-r1}. Recall that $p^*\le 2$ and $a\neq 0$. 
Furthermore, by substituting $-b$ for $b$, we can assume without loss of generality that $a>0$.
Then by writing $t=\frac{b}{a},$ the inequality \eqref{key-ineq-p*-lem-r1} is equivalent to  
    \begin{equation}\label{key-ineq-p*-r1-s1}
        |1+t|^{p^*-2}(1+t)t-t\le (p^*-1+\kappa)\frac{(1+C_1|t|)^{p^*}}{1+|t|^2}|t|^2
    \end{equation}
    for every $t\in \mathbb{R}$ and some $C_1>0.$

    On the one hand, for any $|t|\ll 1,$ applying a Taylor expansion, we get
    $$|1+t|^{p^*-2}(1+t)t=|1+t|^{p^*-1}t=t+(p^*-1)t^2+o(t^2).$$
    Notice that $G(t):=t\mapsto t^{\frac{1}{p^*}}$ is concave since $G''(t)=\frac{1}{p^*}(\frac{1}{p^*}-1)t^{\frac{1}{p^*}-2}\le 0,$ which implies $$1+\frac{1}{p^*}|t|^2\ge (1+|t|^2)^{\frac{1}{p^*}},\quad \text{for all}\ |t|\le 1.$$
    Thus, there exists $t_0=t_0(p^*)>0$ small such that, for any $C_1\ge \frac{1}{p^*}$ and $t\in [-t_0,t_0],$
    $$|1+t|^{p^*-2}(1+t)t\le t+ (p^*-1+\kappa)\frac{(1+C_1|t|)^{p^*}}{1+|t|^2}|t|^2.$$

    On the other hand, when $|t|>t_0,$ in order to prove  \eqref{key-ineq-p*-r1-s1}, it is sufficient to show the existence of $C_1<+\infty,$ for which,
    \begin{equation}\label{key-ineq-r1-s2}\left[\left((1+|t|^2)\frac{|1+t|^{p^*-2}(1+t)t-t}{(p^*-1+\kappa)|t|^2}\right)^{\frac{1}{p^*}}-1\right]|t|^{-1}\le C_1.
    \end{equation}
    Notice that when $|t|\rightarrow +\infty,$ the left hand side of \eqref{key-ineq-r1-s2} is bounded.  This together with the boundedness of a continuous function on a compact set implies that, there exists $C_1<+\infty$ such that, for any $t\in \mathbb{R}^n\setminus (-t_0,t_0),$ \eqref{key-ineq-r1-s2} holds. All in all, we conclude $(1)$ of this lemma.

    \noindent\textbf{Step 2:} Now, we give the proof of \eqref{key-ineq-p*-lem-r2} by similar arguments of \cite[Lemma 3.2]{FN2019}. 
    
    If $a=0,$ inequality \eqref{key-ineq-p*-lem-r2} is trivial since it reduces to $|b|^{p^*}\le  C_{2}|b|^{p^*}.$ Next, we assume $a\ne 0.$
    Suppose  \eqref{key-ineq-p*-lem-r2} fails. Then, there exist $\kappa>0,\{C_{2,j}\}\subset \mathbb{R}$ satisfying $C_{2,j}\rightarrow +\infty$ and $\{a_j\},\{b_j\}\subset \mathbb{R}$ such that 
    \begin{equation}\label{key-ineq-r1-s3}
        |a_j+b_j|^{p^*-2}(a_j+b_j)b_j>|a_j|^{p^*-2}a_jb_j+(p^*-1+\kappa)|a_j|^{p^*-2}|b_j|^2+C_{2,j}|b_j|^{p^*}.  
        \end{equation}
        Furthermore, by substituting $-b_j$ for $b_j$, we can assume without loss of generality that $a_j>0.$ Dividing both sides of \eqref{key-ineq-r1-s3} by $a_j^{p^*}$, we obtain 
        \begin{equation}\label{key-ineq-r1-s4}            \frac{|a_j+b_j|^{p^*-2}(a_j+b_j)b_j}{a_j^{p^*}}>\frac{b_j}{a_j}+(p^*-1+\kappa)\left|\frac{b_j}{a_j}\right|^2+C_{2,j}\frac{|b_j|^{p^*}}{a_j^{p^*}}.
        \end{equation}
        As $C_{2,j}\rightarrow +\infty,$ then it holds $$\frac{|b_j|}{a_j}\rightarrow 0,\quad \text{when}\ j\rightarrow \infty$$
         at a sufficiently fast rate. However, the Taylor expansion yields 
          \begin{equation*}
            \frac{|a_j+b_j|^{p^*-2}(a_j+b_j)b_j}{a_j^{p^*}}=\frac{b_j}{a_j}+(p^*-1)\frac{b_j^2}{a_j^2}+o\left(\frac{b_j^2}{a_j^2}\right),
        \end{equation*}
         which is smaller than the right-hand side of \eqref{key-ineq-r1-s4}. This leads to a contradiction. Thus, we conclude \eqref{key-ineq-p*-lem-r2}. 
\end{proof}


\begin{thebibliography}{99}
\bibitem{A1976}
T. Aubin, \emph{Problemes isop{\'e}rim{\'e}triques et espaces de Sobolev}. Journal of differential geometry
\textbf{11} (1976), 573--598.

\bibitem{BE1991}
G. Bianchi, H. Egnell, \emph{A note on the Sobolev inequality}. Journal of functional analysis
\textbf{100} (1991), 18--24.

\bibitem{BL1985}
H. Brezis, E. Lieb, \emph{Sobolev inequalities with remainder terms}. Journal of functional analysis
\textbf{62} (1985), 73--86.

\bibitem{CFMP2009}
A. Cianchi, N. Fusco, F. Maggi, A. Pratelli, \emph{The sharp Sobolev inequality in quantitative form}. Journal of the European Mathematical Society
\textbf{11} (2009), 1105--1139.

\bibitem{CFM2018}
G. Ciraolo, A. Figalli, F. Maggi, \emph{A quantitative analysis of metrics on with almost constant positive scalar curvature, with applications to fast diffusion flows}. International mathematics research notices
\textbf{2018} (2018), 6780--6797.

\bibitem{CG2025}
G. Ciraolo, M. Gatti, \emph{On the stability of the critical $p$-Laplace equation}. arXiv preprint arXiv:2503.01384 
 (2025).

\bibitem{CNV2004}
D. Cordero-Erausquin, B. Nazaret, C. Villani, \emph{A mass-transportation approach to sharp Sobolev and Gagliardo--Nirenberg inequalities}. Advances in Mathematics
\textbf{182} (2004), 307--332.

\bibitem{DSW2021}
B. Deng, L. Sun, J. Wei, \emph{Sharp quantitative estimates of Struwe's decomposition}. Duke Mathematical Journal
\textbf{174} (2025), 159--228.

\bibitem{DEFFL2022}
J. Dolbeault, M. Esteban, A. Figalli, R. Frank, M. Loss, \emph{Sharp stability for Sobolev and log-Sobolev inequalities, with optimal dimensional dependence}. arXiv preprint arXiv:2209.08651
(2022).
 
\bibitem{FG2020}
A. Figalli, F. Glaudo, \emph{On the sharp stability of critical points of the Sobolev inequality}. Archive for rational mechanics and analysis
\textbf{237} (2020), 201--258.

\bibitem{FMP2013}
A. Figalli, F. Maggi, A. Pratelli, \emph{Sharp stability theorems for the anisotropic Sobolev and log-Sobolev inequalities on functions of bounded variation}. Advances in Mathematics
\textbf{242} (2013), 80--101.

\bibitem{FN2019}
A. Figalli, R. Neumayer, \emph{Gradient stability for the Sobolev inequality: the case $p\geq 2$}. Journal of the European Mathematical Society (EMS Publishing)
\textbf{21} (2019), 319--354.

\bibitem{FZ2022}
A. Figalli, Y. Zhang, \emph{Sharp gradient stability for the Sobolev inequality}. Duke Mathematical Journal
\textbf{171} (2022), 2407--2459. 

\bibitem{MW2010}
C. Mercuri, M. Willem, \emph{A global compactness result for the p-Laplacian involving critical nonlinearities}. Discrete and continuous dynamical systems
\textbf{28} (2010), 469--493.

\bibitem{N2020}
R. Neumayer, \emph{A note on strong-form stability for the Sobolev inequality}. Calculus of Variations and Partial Differential Equations
\textbf{59} (2020), 25.

\bibitem{S1984}
M. Struwe, \emph{A global compactness result for elliptic boundary value problems involving limiting nonlinearities}. Mathematische zeitschrift
\textbf{187} (1984), 511--517.


\bibitem{T1976}
G. Talenti, \emph{Best constant in Sobolev inequality}. Annali di Matematica pura ed Applicata
\textbf{110} (1976), 353--372.

\bibitem{ZZ2023}
Y. Zhou, W. Zou, \emph{Quantitative stability for the Caffarelli-Kohn-Nirenberg inequality}. arXiv preprint arXiv:2312.15735
 (2023).
\end{thebibliography}
\end{document}